\documentclass[10pt]{article}
  \usepackage[centertags]{amsmath}
  \usepackage{amsfonts}
  \usepackage{amssymb}
  \usepackage{amsthm}
  \usepackage{newlfont}

  \newlength{\defbaselineskip}
  \newcommand{\setlinespacing}[1]%
                               {\setlenght{\baselineskip}{#1 \defbaselineskip}}

  \newcommand{\re}{Re}

  \theoremstyle{plain}
  \newtheorem{thm}{Theorem}[section]
  \newtheorem{cor}[thm]{Corollary}
  \newtheorem{lem}[thm]{Lemma}
  
  \newtheorem{rem}[thm]{Remark}
  
  \theoremstyle{definition}
  \newtheorem{defi}[thm]{Definition}

  \numberwithin{equation}{section}

\begin{document}
\begin{center}
{\bf Best Approximation in Numerical Radius}
\end{center}
\vspace{.15 cm}
\begin{center}
\small{Asuman G\"{u}ven AKSOY  and Grzegorz LEWICKI}
\end{center}

\date{January 7, 2010}
\mbox{~~~}\\
\mbox{~~~}\\
\small\mbox{~~~~}{\bf Abstract.} {\footnotesize Let $X$ be a reflexive Banach space. In this paper we give a necessary and sufficient condition for an operator $T\in \mathcal{K}(X)$ to have the best approximation in numerical radius from the convex subset $\mathcal{U} \subset \mathcal{K}(X),$ where $\mathcal{K}(X)$ denotes the set of all linear, compact operators from $X$ into $X.$  We will also present  an application to minimal extensions with respect to the numerical radius. In particular some results on best approximation in norm will be generalized to the case of the numerical radius.} \\

\
\\

\section{ Introduction.}
Let $X$ be a Banach space over $\mathbb{R}$ or $\mathbb{C}$, we use $B_{X}$ for
the closed unit ball and $S_{X}$ for the unit sphere of $X$. The
dual space is denoted by $X^{*}$ and the Banach algebra of all
continuous linear operators on $X$ is denoted by $B(X)$. The
\textit{numerical range} of $T\in B(X)$ is defined by
$$W(T)= \{ x^{*}(Tx)  :~x\in S_{X},~x^{*}\in S_{X^{*}},~x^{*}(x)=1\}\cdot$$
The \textit{numerical radius} of $T$ is then given by
$$\parallel T \parallel_{w}=\sup\{\vert \lambda\vert : ~\lambda\in W(T)\}\cdot$$
Clearly, $\parallel . \parallel_{w}$ is a semi-norm on $B(X)$ and $\parallel T \parallel_{w} \le \Vert T\Vert$ for
all $T\in B(X)$.  Therefore, there are more operators for which $\Vert T\Vert\geq 1$ but $\parallel T \parallel_{w}= 1$ and this gives a motivation to study both the minimal projections and the  best approximation with respect to $\parallel \cdot  \parallel_{w} $. Naturally for $\Vert T\Vert$ we are taking ``the supremum" over $(x,x^*) \in B(X)\times B(X^*)$, but for $\parallel T \parallel_{w}$ ``the supremum" is taken over those $(x,x^*)$ for which $x^*(x) =1$.

 The \textit{numerical index} of $X$ is defined by
$$n(X)=\inf\{\parallel T \parallel_{w} :~ T\in S_{B(X)}\}\cdot$$
Equivalently, the numerical index $n(X)$ is the largest constant $k
\geq 0$ such that $$k\|T\| \leq \parallel T \parallel_{w}$$ for every $T \in B(X)$. Note
also that $0 \leq n(X) \leq 1$, and $n(X) > 0$ if and only if
$\parallel \cdot \parallel_{w}$
and $\|\cdot\|$ are equivalent norms.

The concept of numerical index was first introduced by Lumer
\cite{lg} in 1968. Since then much attention has been paid to this
constant of equivalence between the numerical radius and the usual
norm in the Banach algebra of all bounded linear operators of a
Banach space. Classical references here are  \cite{bff-dj1},
 \cite{bff-dj2}. For recent results we refer the reader to
\cite{aag-cbl}, \cite{aga-ed-kham}, \cite{ee},
\cite{fc-mm-pr},\cite{gke-rdkm}, \cite{lg-mm-pr}, \cite{mm}.

Existence and uniqueness of best approximation in particular subsets of $ \mathcal{U}\subset B(X)$ in the operator norm is one of the basic questions of approximation theory. One
very important case of $\mathcal{U}$ is a set of all linear continuous projections from a Banach space $X$ onto its subspaces $Y.$ More precisely,
let $ Y \subset X$ be a linear, closed subspace. A linear map $P: X \rightarrow Y$ is called a \emph{projection} if $Py=y$ for any $y \in Y.$ Clearly, if $Y \neq \{ 0\}$, then $\| P \| \geq 1$ for any projection $P$. The set of all projections going from $X$ onto $Y$ will be denoted by $\mathcal{P}(X,Y)$.
 Minimal projections play a special role among all projections. A projection $P_o \in \mathcal{P}(X,Y)$ is called \emph{minimal} if
$$
\|P_o\| = \inf \{ \|P\|: P\in \mathcal{P}(X,Y)\} = dist(0,\mathcal{P}(X,Y))\}.
$$
There is a lot of previous research concerning minimal projections. Primarily this work addresses problems of finding minimal projections effectively and estimating  norms of minimal projections
and uniqueness of minimal projections. (e.g \cite{aag-cbl}, \cite{chalL1}-\cite{CMO}, \cite{FMW}, \cite{IS1}, \cite{K1}-\cite{LEP}, \cite{LES}, \cite{Li}, \cite{LZ1}, \cite{OL}, \cite{vpo}, \cite{SK1}-\cite{SK3}, \cite{WO1}).\\

Now suppose that $V$ is a subset of a Banach space $X$ and  $x_0 \in X \setminus V$. Denote by $P_V(x_o)$ the set of all best approximants to $x_0$ in  $V.$ We say $v_0 \in V$ is a \emph{strongly unique best approximation} (SUBA) to $x_0$ if and only if there exists $r > 0$ such that for all $u\in V$
$$
\| x_0-u \| \geq\| x_0-v_0\| + r \| u-v_0 \| .
$$
It is clear that if $v_0$ is a SUBA then $v_o \in P_V(x_o)$. It is also easy to see that $v_o$ is the only element of best approximation. There are natural examples of SUBA. Here we mention the most important one. Let $X=C[0,1]$ and $V_n$ be the subspace of polynomials of degree less than or equal to $n$. If $f$ is any element of $ C[0,1]$ and   $P_0 \in \mathcal{P}_{V_{n}}(f)$ then $P_0$ is a SUBA to $f$.

 Also strong unicity can be applied in the proof of the SUBA Theorem (\cite{ewc}, p.80 in the case of polynomial approximation) concerning the Lipschitz
continuity of the best approximation operator. More precisely, let $V \subset X $ and  $(x_n) \in X$ with $x_n \rightarrow x$. Suppose $P_{V}(x_n)$ is a best approximation to $x_n$ in $V$ and $x$ has a SUBA element $P_{V}(x)$ with the constant $ r>0.$ Then
$$ \| P_{V}(x_n)-P_{V}(x)\| \leq \displaystyle \frac{2}{r} \,\,  \| x_n -x \|.$$
Also, the strong unicity constant plays a crucial role in the estimate of the error of the Remez algorithm (see \cite{ewc}, p.97).
For further details concerning strong unicity we refer to \cite{brw}, \cite{ewc}, \cite{LE0}, \cite{LED}, \cite{SW}.

The aim of this paper is to prove some criteria for best approximation and SUBA with respect to the numerical radius and some related seminorms. More precisely,
let $X$ be a reflexive Banach space (we consider both the real and the complex cases) and let $ \mathcal{K}(X)$ denote the set of all compact oparators from $X$ into $X.$ Let us consider
\begin{equation}
\label{B}
 \mathcal{B} =B_{X^*} \times B_{X}
\end{equation}
with the Tychonoff topology induced by the
weak$^*$-topology in $B_{X^*}$ and by the weak-topology in $B_X.$ By the Banach-Alaoglu Theorem and the Tychonoff Theorem, $\mathcal{B}$ is a compact set. Assume that $W \subset \mathcal{B}$ is a fixed, non-empty and compact subset of $\mathcal{B}.$ Define for $ L \in \mathcal{K}(X)$
$$
\| L\|_W = \sup \{ |x^*(Lx)|:(x^*,x) \in W \}.
$$
It is clear that $\| \cdot \|_W$ is a semi-norm on $\mathcal{K}(X).$ Let
$$
\mathcal{W}(X) = \mathcal{K}(X)/(R),
$$
where $(R)$ is an equivalence relation on $\mathcal{K}(X) \times \mathcal{K}(X)$ defined by
$$
L\hbox{ } (R)\hbox{ }T \hbox{ if and only if } \|L-T\|_W = 0.
$$
Note that $\mathcal{W}$ becomes a Banach space with the norm
$$
\| [L]\|_W = \|L\|_W.
$$
Here the symbol $ [L]$ denotes the equivalence class of $L$ with respect to $(R).$

In this paper we prove some criteria for best approximation and SUBA  in the quotient space $ \mathcal{W}(X)$ where $X$ is a reflexive space. Also an application to minimal extensions with respect to the numerical radius will be presented. It is worth noticing that  \cite{aag-cbl} gives a characterization of minimal numerical-radius extensions of operators from a normed linear space $X$ onto its finite dimensional subspaces and
comparison with minimal operator-norm extension.

We will use the following results throughout the paper. Let $X$ be a Banach space and let $ext(S_{X^*})$ denote the set of all extreme points of $S_{X^*}.$ For any $x \in X$ set
$$
E(x)= \{ f \in ext(S_{X^*}):\,\, f(x)= \| x\| \}.
$$
We have $E(x) \neq \emptyset$ for any $ x \in X$, by the Banach-Alao\u{g}lu Theorem and the Krein-Milman Theorem.
\begin{thm}
\label{bros}
\cite{brw} Let $V \subset X$ be a convex set and let $x_o \in X.$ Then $v_o \in V$ is a best approximation to $x_o $ in $V$ if and only if for any $ v \in V$ there exists
$f \in E(x-v_o)$ with
$$
re(f(v-v_o))\leq 0.
$$
If $V$ is a linear subspace, the above inequality can be replaced by
$$
re(f(v))\leq 0.
$$
Here for $ z \in \mathbb{C},$ the symbol $re(z)$ denotes the real part of $z.$
\end{thm}
\begin{thm}
\label{wojcik}
\cite{SW}
 Let $V \subset X$ be a convex set and let $x_o \in X.$ Then $v_o \in V$ is a SUBA to $x_o $ in $V$ with $ r > 0$ if and only if for any $ v \in V$ there exists
$f \in E(x-v_o)$ with the following:
$$
re(f(v-v_o))\leq -r\|v-v_o\|
$$
If $V$ is a linear subspace the above inequality can be replaced by
$$
re(f(v))\leq -r\|v\|.
$$
\end{thm}
\section{Main Results }
In the complex case define for any $ \theta \in [0,2\pi]$
$$
W_{\theta} = \{ (e^{i\theta}x^*,x): (x^*,x) \in W \}
$$
and
\begin{equation}
\label{Z}
Z = \bigcup_{\theta \in [0,2\pi]}W_{\theta}.
\end{equation}
Set for any $ T \in \mathcal{W}(X)$
\begin{equation}
\label{set}
W_T= \{ (x^*,x) \in Z :\,\,\, x^*(Tx) = \| [T]\|_{W} \,\,\}.
\end{equation}
Observe that the above definition does not depend on a particular representation of $[T].$
To define $W_T$ in the real case we should replace the set $Z$ by
\begin{equation}
\label{ZR}
Z_{\mathbb{R}} = W \cup \{(-x^*,x): (x^*,x) \in W\}.
\end{equation}
We start with
\begin{lem}
\label{nonempty}
For any $ T \in \mathcal{W},$ $W_T \neq \emptyset.$
\end{lem}
\begin{proof}
Fix $ [T] \in \mathcal{W}(X).$
First we consider the complex case.
Since $W$ is a compact set, the set $Z$ defined by (\ref{Z}) is a compact set too.
Fix $ L \in [T].$ Define a function $ \Phi(L):Z \rightarrow \mathbb{C}$ by
$$
\Phi(L)((e^{i\theta}x^*,x))= e^{i\theta}(x^*Lx).
$$
Now we show that $\Phi(L)$ is a continuous function.

To do this, fix a net $\{z_{\gamma}=(x^*_{\gamma},x_{\gamma})\}\subset Z$ tending to $z=(x^*,x).$ Assume on the contrary that $ \Phi(L)(z_{\gamma})$ does not converge to
$\Phi(L)(z).$ Notice that $Lx_{\gamma} \rightarrow Lx$ in the weak topology.
Since $L$ is a compact operator, passing to a subnet, if necessary, we can assume that $ \|Lx_{\gamma} -Lx\|\rightarrow 0.$ Consequently,
$$
|x^*_{\gamma}Lx_{\gamma} - x^*Lx| \leq |x^*_{\gamma}(Lx_{\gamma}-Lx) + (x^*_{\gamma}-x^*)Lx|
$$
$$
\leq \|Lx_{\gamma}-Lx\| + |(z^*_{\gamma}-z^*)Lx| \rightarrow_{\gamma} 0.
$$
which is a contradiction. Since $\Phi(L)$ is continuous and $Z$ is a compact set there exists $ z_o \in Z$ such that
$$
\Phi(L)(z_o) =\sup\{|\phi(L)(z)|: z \in Z\},
$$
which shows that $ W_T \neq \emptyset,$ as required.
\newline
The proof in the real case goes exactly in the same way with $Z$ replaced by $Z_{\mathbb{R}}$ defined by (\ref{ZR}).
\end{proof}
\begin{thm}
\label{characterization}
Let $X$ be a reflexive Banach space and $W \subset \mathcal{B}$ be a fixed, non-empty, compact subset of $\mathcal{B}.$ Let $ \mathcal{U} \subset \mathcal{W}(X)$ be a non-empty  convex subset of $ \mathcal{W}(X). $ An element $L \in \mathcal{U}$ is a best approximation to $T \in \mathcal{W}(X)$ if and only if for any $U \in \mathcal{U}$ there exists $(x^*,x) \in W_T$ such that
\begin{equation}
\label{un}
re(x^*(U-L)x)\leq 0.
\end{equation}
\end{thm}
\begin{proof}
First we consider the complex case. Take $Z$ defined by (\ref{Z}).
Let $C(Z)$ denote the space of all continuous, complex-valued or real-valued functions defined on $Z$ equipped with the supremum norm $ \| \cdot \|_{sup}.$
Let $ \Phi:\mathcal{W}(X)\rightarrow C(Z)$ be defined by
$$
\Phi([L])(x^*,x)= x^*Lx
$$
for any $(x^*,x) \in Z.$
Reasoning as in Lemma (\ref{nonempty}) we can show, applying compactness of $L,$ that $\Phi{[L]}$ is a continuous function on $Z,$ where $Z$ is endowed with the topology induced from $ \mathcal{B}$ given by (\ref{B}). Moreover, $ \Phi$ is a linear
isometry. Consequently $L$ is a best approximation to $T$ in $\mathcal{U}$ if and only if $\Phi(L)$ is a best approximation to $\Phi(T)$ in $ \Phi(\mathcal{U}).$
By Theorem (\ref{bros}) and the form of extreme points of the unit sphere in $ C^*(Z),$ this is equivalent to the fact that for any $ \Phi(U) \in \Phi(\mathcal{U})$ there exist $(x^*,x) \in Z$ such that
$\Phi(T-L)(x^*,x)= \|\Phi(T-L)\|_{sup}$ and
$$
\re((\Phi(U)-\Phi(L))(x^*,x))= re(x^*(U-L)x) \leq 0,
$$
which completes the proof in the complex case. The proof in the real case goes in the same way with $Z$ replaced by $Z_{\mathbb{R}}$ given by (\ref{ZR}).
\end{proof}
Applying Theorem (\ref{wojcik}) and a similar reasoning used in Theorem (\ref{characterization}) we can prove:
\begin{thm}
\label{stronguni}
Let $X$ be a reflexive Banach space and $W \subset \mathcal{B}$ be a fixed, non-empty, compact subset of $\mathcal{B}.$ Let $ \mathcal{U} \subset \mathcal{W}(X)$ be a non-empty  convex subset of $ \mathcal{W}(X). $ An element $L \in \mathcal{U}$ is a SUBA to $T \in \mathcal{W}(X)$ with $r >0$ if and only if for any $U \in \mathcal{U}$ there exists $(x^*,x) \in W_T$ such that
$$
x^*(U-L)x \leq -r \| U - L \|_{w}
$$
\end{thm}
\begin{rem}
\label{important1}
Note that in Theorem (\ref{characterization}) and Theorem (\ref{stronguni}) we can replace set $W_T$ by
$$
E_T = W_{T} \cap (ext(S_{X^*})\times ext({S_X})).
$$
\end{rem}
Indeed let $(x^*,x) \in W_T$ satisfy $ re(x^*(U-L)x) \leq 0.$ Set
$$
N_x = \{ z^* \in B_{X^*}:z^*(x) = \|x\| =1\}.
$$
It is clear that $N_x$ is a nonempty, convex set and that $ext(W_x) \subset ext(S_{X^*}).$ Consider a function
$$
g(w^*) = re(w^*(U-L)x).
$$
Since $g$ is a linear functional on $X^*$ and $ re(x^*(U-L)x) \leq 0,$ there exists $z^*$ in  $ ext(S_{X^*})= ext(W_x)$ with
$$
re(z^*(U-L)x) \leq 0.
$$
Now set$$
N_{z^*} = \{ x \in B_{X}:z^*(x) = \|z^*\| =1\}.
$$
Since $X$ is reflexive, by the James Theorem $N_{z^*} \neq \emptyset.$ Reasoning as we did above, we get that there exists $ z \in ext{S_X}$ satisfying $z^*(z)=1$ with
$$
re(z^*(U-L)z) \leq 0,
$$
which shows our claim in the case of Theorem (\ref{characterization}). The same reasoning works in the case of Theorem (\ref{stronguni}).

\begin{rem}
\label{important2}
Note that in Theorem (\ref{characterization}) and Theorem (\ref{stronguni}) we can replace $\mathcal{K}(X)$ by any subspcace $\mathcal{D}$ of
$\mathcal{K}(X).$ In this case the equivalence relation $(R)$ should be replaced by its restriction to $ \mathcal{D} \times \mathcal{D}.$
\end{rem}
\begin{cor}Assume that $X$ is a finite-dimensional space. For any number $q$ with $ 0 <q \leq 1,$  set
$$
W_q = \{ (x^*,x): x^*(x)=q\}.
$$
Also define ``q-numerical range" for $ T \in B(X)$ by
$$
W_q(T)= \{ x^*(Tx): \| x \|= \| x^* \|=1\,\, (x,x^*)=q \,\, \}
$$
and
$$
\| T \|_{W_q} = \sup \{|\lambda |:\lambda\in W_q(T)\}.
$$
Then the conclusion of Theorems (\ref{characterization}) and Theorem (\ref{stronguni}) remain true for the best approximation in $\mathcal{W}(X)$ with respect to
$\|\cdot\|_{W_q}$. If we put $q=1$ we get criteria for the best approximation with respect to the numerical radius.
\end{cor}
\begin{proof}
Since $X$ is finite-dimensional, the set $W_q$ is a compact subset of $\mathcal{B}.$ Hence Theorem (\ref{characterization}) and Theorem (\ref{stronguni}) can be applied
to $\| \cdot \|_{W_q}.$
\end{proof}
 For more details on \emph{q-numerical range} we refer to \cite{ma}.
\begin{rem}
\label{rem1}
If $X$ is reflexive and $ W = \mathcal{B}=B_{X^*} \times B_X,$ Theorem (\ref{characterization}) and Theorem (\ref{stronguni}) have been proved in \cite{OL})(see also \cite{LED}).
\end{rem}
\section{An Application}
Investigating minimal projections in $ \mathcal{P}(X,V) \subset B(X)$ with respect to various semi-norms on $ B(X)$ raises the question of on what subspaces of
$ B(X)$ semi-norms are actually norms. The following lemma provides an answer to this question in the case of the numerical radius $ \| \cdot \|_w.$
\begin{lem}
\label{lem1}
Let $X$ be a Banach space, $V$ its $n$-dimensional subspace, and
$$
B_V (X,V) = \displaystyle \{ L \in B(X,V):\,\, L|_V = 0 \}.
$$
Let for $ A \in B(V),$
$$
\|A\|_w = \sup\{|v^*Av|:(v^*\in B_{V^*},V\in B_V, v^*(v)=1\}.
$$
Suppose $A \in B(V) \setminus \{0\}$ with $\| A \| _w > 0$ and $A_0 \in B(X,V)$ with $A_0|_V = A$  a fixed operator. Consider a subspace
$$
Z_A\subset B_V (X,V)
$$
defined by
$$ Z_A= span[A_0]\oplus B_V (X,V)$$ where by $span[ A_0]$ we mean the subspace spanned by $A_0$.  Then the semi-norm  $\| . \|_w$ defined with respect to the subspace $Z_A$ is actually a norm on $Z_A$.
\end{lem}
\begin{proof}
Let $L \in Z_A \setminus \{0\}$; we want to show $\| L \|_w > 0$.  Since $L \in Z_A$, then $L= \alpha A_0 + L_1$ where $\alpha \in \mathbb{R}$ and $L_1 \in B_V (X,V)$.\\
Case 1. Assume $\alpha\neq 0$ :\\
From our assumption $\| A \| _w > 0$, we know that for some  $v \in S(V)$ and $v^* \in S(V^*)$ with $v^*(v)=1$ we have $ | v^* A v | > 0$. Let $x^* \in S_{X^*}$ be the Hahn-Banach extension of $v^*$  to $X$, then
$$
x^* L v  = \alpha x^* A_0 v + x^* (L_1v).
$$
Since $L_1 v =0$ and $A_0|_V =A,$ $x^* L v = \alpha  x^* A v $ with $\alpha\neq 0$ and  $ | v^* A v | > 0$ and therefore $\| L \|_w > 0$.\\
 Case 2. Assume  $\alpha =0$ :\\
 Let $L \in B_V (X,V)\setminus \{0\}$ and set $L = \displaystyle \sum_{i=1}^{k} f_i(\cdot) v_i$ where $v_1,v_2, \cdots, v_k \in V\setminus \{0\} $ and
$f_1,f_2,\cdot,f_k\in X^*$
are such that
 \begin{itemize}
 \item $f_i |_V = 0$ for $i= 1, 2, \cdots k$
 \item $\{f_i\}_{i=1}^{k}$ is a linearly independent set.
 \end{itemize}
Let
$$
X_1= \displaystyle \bigcap _{i=1}^{k} \ker(f_i) \,\,\,\mbox{and} \,\,\, X_2 = \displaystyle \bigcap _{i=2}^{k} \ker(f_i).
$$
(We put $ X_2 = X$ if $k=1$).

Since $\{f_i\}_{i=1}^{k}$ is  linearly
independent, we know $V \subset X_1 \varsubsetneq X_2$. Fix $x \in X_2 \setminus X_1$ such that $ 0 \notin \mathcal{P}_{V_1}(x)$ where $V_1 =span[v_1]$. By $\mathcal{P}_{V_1}(x)$ we mean the set of  best approximation to $x$ from $V_1$. Without loss of generality assume $\| x \| = 1 $.

 Then by the Hahn-Banach Theorem  for any $x^* \in S(X^*)$ with $x^*(x)=1$, we have
 $x^*(v_1) \neq 0$ and hence $$ x^*Lx = x^* (\displaystyle \sum_{i=1}^{k} f_i(x) v_i) = x^* (f_1(x))v_1= f_1(x) x^*(v_1) \neq 0 $$ giving again $\| L \|_w > 0$.
\end{proof}
\begin{rem}
\label{rem3}
Note that $ \|Id_V \|_w=1.$ Hence  $\| . \|_w $ is actually a norm in restriction to $Z_{Id_V}.$
\end{rem}
In  \cite{vpo} it was shown that for any three dimensional real Banach space $X$ and any of its two dimensional subspace $V$ if the infimum with respect to the  operator norm
over $\mathcal{P}(X,V)$ is greater than one, then there exists the unique projection of minimal operator norm. Later in \cite{LE0} (see also \cite{LED})  this result was generalized as follows:
\\ Let $X$ be a three dimensional real Banach space and $V$ a two dimensional subspace of $X$. Suppose  $A \in B(V)$ is a fixed operator. Set
$$
\mathcal{P}_{A}(X,Y)= \{P\in B(X,Y) : \,\,\, P\mid_{Y}=A\,\,\}
$$
and  assume $\parallel P_0\parallel> \parallel A \parallel$, where $P_0 \in \mathcal{P}_{A} (X,Y)$ is an extension of minimal operator norm. Then $P_0$ is a SUBA minimal extension with respect to the operator norm.\\ In other words  for all $P \in \mathcal {P} _{A} (X,Y)$ one has
$$\| P \|\geq \| P_0 \| +\,r \,\| P-P_0\| $$
\begin{defi}
We say an operator $0$ is a SUBA to $A_o$ with respect to numerical radius in $B(X)$ if $A_0|_V = A $  and there exists $r > 0$ such that
$$
\| B \|_w \geq \| A_0 \|_w +\,r \,\| B-A_0\|_w
$$
for any $B \in B(X,V)$ with $B|_V =A.$
\end{defi}
A natural extension of the above result to $\| \cdot \|_{w}$ is as follows:
\begin{thm}
\label{application}
Assume that $X$ is a three dimensional real Banach space and let $V$ be a two dimensional subspace of $X$, and that $A \in B(V)$ with $\| A \|_w >0.$ Let
$$
\lambda_{w}^{A} = \lambda_{w}^{A}(X,V) = \inf \{ \| A_{0} \|_w : \,\,\, A_0 \in B(X,V)\, \,\,\,\,\, A_0|_V = A \} > \|A \|,
$$where $\|A\|$ denotes the operator norm.
Then there exists exactly one $A_0 \in B(X,V)$ such that $A_0|_V = A$ and $$\lambda_{w}^{A} = \| A_0 \|_w .$$ Moreover, $0$ is a SUBA to $A_o$ with respect to numerical radius in  $B_{V} (X,V)$.
\end{thm}
\begin{proof}
Since $\|A\|_w >0,$ by Lemma(\ref{lem1}) $\| \cdot \|_w$ is a norm on $Z_A.$
Since $X$ is finite-dimensional, any operator $L \in Z_A$ posseses a best approximation in $B_V(X,V)$ with respect to the $\| \cdot \|_w.$  Hence there exists $A_o \in \mathcal{P}_A(X,V)$ such that $ \|A_o\|_w = \lambda_w^A.$
Let $W_{A_o}$ be defined by (\ref{set}). Set for any $(x^*, x) \in X^* \times X$ and $ L \in B(X)$
$$
(x^*\otimes x)(L) = x^*(Lx).
$$
Note that  $ x^* \otimes x$ is a linear, continuous functional on $ B(X)$  for any $ (x^*,x) \in X^* \times X.$
Let
$$
C = \{ x^*\otimes x : (x^*,x) \in W_{A_o}\}.
$$
First we show that $ 0 \in conv(C|_{B_V(X,V)}).$ Assume that this is not true. Since $X$ is finite-dimensional and $C$ is a compact set, by the Carath\'eodory Theorem
(see \cite{ewc})  $conv(C|_{B_V(X,V)})$ is also a compact set.  Since $ 0 \notin conv(C|_{B_V(X,V)}),$ by the Separation Theorem there exists $ L \in B_V(X,V)$ such that
$$
(x^* \otimes x)(L)= x^*(Lx) > 0
$$
for any $(x^*,x) \in W_{A_o}.$ By Theorem (\ref{characterization}) applied to $ A_o$ and $ B_V(X,V),$ it follows that $A_o$ is not a minimal extension of $A$ which is a contradiction.
Consequently,
\begin{equation}
\label{conv}
0 = \sum_{j=1}^k a_j(x^*_j \otimes x_j)|_{B_V(X,V)},
\end{equation}
where $ a_j > 0$ and $ \sum_{j=1}^k a_j=1.$ Let
$
k_o = \min\{ k: k \hbox{ satisfies }\ref{conv} \}.
$
Note that \\$dim(B_V(X,V))=2,$ since $dim(X)=3,$ and $dim(V)=2.$ Hence by the Carath\'eodory Theorem, (see \cite{ewc}),we conclude that $k_o \leq 3.$

Now we show that $k_o=3.$ Assume this is not true. If $ k_o =1, $ then \\$ (x^* \otimes x)|_{B_V(X,V)} = 0 $ for some $ (x^*, x) \in W_{A_o}.$
Fix $ f \in X^* \setminus \{0\}$ satisfying $ V = ker(f).$ Since $$ \|A_o\|_w > \|A\| \geq \|A\|_w $$ it follows that
$ f(x) \neq 0.$ Take $ L \in B_V(X,V)$ given by $ Lz = f(z)A_ox.$ Note that
$$
x^*(Lx) = f(x)x^*(A_ox) = f(x) \|A_o\|_w \neq 0,
$$
which leads to a contradiction.
Now assume that $k_o=2.$ Then
\begin{equation}
\label{two}
0 = a_1(x^* \otimes x)|_{B_V(X,V)} + a_2 (y^* \otimes y)|_{B_V(X,V)}
\end{equation}
where $a_1 > 0, a_2 >0 $ and $a_1 + a_2 =1.$
First we show that $x$ and $y$ are linearly independent. If not, since $ \| x\| = \|y\|=1,$ we have $x=y$ or $x=-y.$
By (\ref{two}) taking \\$ L = f(\cdot)A_ox$ we get
$$
0= a_1f(x)\|A_o\|_w + (1-a_1)f(x) \|A_o\|_w
$$
which gives $ f(x) = 0.$ Hence $ \|A_o\|_w = \|A\|_w \leq \|A\|$ which is a contradiction.

Now we show that $ x^*|_V = b y*|_V$ for some $ b \neq 0.$ Note that, since $ x^*(A_ox)= y^*(A_oy) = \|A\|_o,$ we have  $ x^*|_V \neq 0$ and $ y^*|_V \neq 0.$
If $x^*|_V$ and $ y*|_V$ were linearly independent, then we could find $v_1\in V$ such that $x^*(v_1)=1$ and $ y^*(v_1)=0.$
Set $ S= f(\cdot)v_1.$
By (\ref{two}) applied to $S$ we get
$$
0 = a_1f(x)x^*(v_1)= a_1f(x).
$$
Since $f(x)\neq 0,$ it follows that $a_1=0$ is a contradiction.
By (\ref{two}) applied to $ L = f(\cdot)A_oy$ we get
$$
0=a_1b f(x)y^*(A_oy) +(1-a_1)f(y)y^*(A_oy) = \|A_o\|_w (f(a_1bx +(1-a_1)y)
$$
and consequently $ f(a_1bx +(1-a_1)y)=0.$ Since $f(x) \neq 0 $ and $ f(y)\neq 0,$ we can find exactly one $ c_1 >0 $
such that $ f(c_1x + (1-c_1)y) = 0$ if $ f(x)f(y) <0 $ or \\ $ f(c_1(-x) + (1-c_1)y) = 0$ if $ f(x)f(y) >0.$

 Since $x$ and $ y$ are linearly independent we get that,
$ b=1$ if $ f(x)f(y) >0$ and $ b=-1$ if $ f(x)f(y) <0.$
Consequently,

$$
y^*(A_o(c_1x +(1-c_1)y)=\|A_o\|_w
$$
if $ f(x)f(y) >0$ and
$$
y^*(A_o(c_1(-x) +(1-c_1)y)=\|A_o\|_w
$$
if $f(x)f(y) < 0.$
But this leads to $ \|A_o\|_w \leq \|A\|$ which is a contradiction.
Hence we have proved that $ k_o =3.$ By (\ref{conv}) we get
\begin{equation}
\label{conv3}
0 = \sum_{j=1}^3 a_j(x^*_j \otimes x_j)|_{B_V(X,V)},
\end{equation}
where $ a_i > 0,$ for $i=1,2,3$ and $ a_1 +a_2+a_3=1.$

Now we show that for any $ i_1,i_2 \in \{ 1,2,3\}, $ $ i_1 \neq i_2,$ it follows that\\ $ g_1 = (x^*_{i_1}\otimes x_{i_1})|_{B_V(X,V)}$ and $ g_2=(x^*_{i_2}\otimes x_{i_2})|_{B_V(X,V)} $ are linearly independent. Without loss of generality we can assume that $ i_1 = 1 $ and $ i_2 =2.$
If not, there exists $ a, b \in \mathbb{R}$ such that $ |a| + |b|> 0 $ and
\begin{equation}
\label{conv4}
ag_1 + bg_2 =0.
\end{equation}
Since $k_o =3,$ we have  $ a \neq 0,$ $b \neq 0$ and $ ab < 0.$
Without loss of generality, we can assume that $ a >0.$ Multiplying (\ref{conv4}) by $ -a_2/b$ and adding it to (\ref{conv3}) we get
$$
((a_1+a(-a_2/b)(x^*_1\otimes x_1) + a_3(x^*_3\otimes x_3))|_{B_V(X,V)} = 0.
$$
Since $ -a_2/b >0,$ and $ k_o=3,$ we get a contradiction, so $g_1$ and $g_2$ are linearly independent.

Now take any $ L \in B_V(X,V)$, define with $ \| L \|_w=1.$ Since $g_1$ and $g_2$ are linearly independent and by (\ref{conv3}) there exists $ i \in \{1,2,3\}$ such that
$ (x^*_i \otimes x_i)(L) = x^*_i(Lx_i) < 0.$ For $ L \in B_V(X,V), $ with $ \| L \|_w=1$, define
$$
g(L) = min \{ (x^*_i \otimes x_i)L :i=1,2,3\}.
$$
It is clear that $g$ is a continuous function on $ S_{B_V(X,V)}$ and $ g(L) < 0 $ for any \\$ L \in S_{B_V(X,V)}.$ Since $X$ is finite-dimensional, $S_{B_V(X,V)}$ is a compact set and
$$
s = \sup \{ g(L) : L \in S_{B_V(X,V)} \} < 0.
$$
Now take any $ L \in B_V(X,V) \setminus \{ 0 \}.$ Then there exists $ i \in \{ 1,2,3\}$ such that
$$
g(L/\|L\|_w) =(x^*_i \otimes x_i)(L/\|L\|_w) \leq s.
$$
 Theorem (\ref{stronguni}) implies that  $ 0 $ is a SUBA to $ A_o, $ with $ r =-s$, and the proof is complete.
\end{proof}
Notice if we take $ A = id_V$ then $ \|A\|_w = \|A\|_o =1.$ In this situation Theorem (\ref{application}) takes the following form.
\begin{thm}
\label{projections}
Assume that $X$ is a three dimensional real Banach space and let $V$ be its two dimensional subspace. Assume that
$$
\lambda^{id_V}_w > 1.
$$
Then there exists exactly one $P_0 \in \mathcal{P}(X,V)$ of minimal norm. Moreover $0$ is a SUBA to $P_o$ with respect to  the numerical radius in  $B_{V} (X,V)$. In particular $P_o$ is the only  minimal projection with respect to the numerical radius.
\end{thm}
\begin{rem}
\label{normone} In Theorem (\ref{application}) the assumption that $ \|A\| < \lambda_A(X,V)$ is essential.
 \end{rem}

 Indeed, let $ X= l_{\infty}^{(3)},$ $ V =\{x \in X: x_1 +x_2=0\}$ and
$ A= id_V.$ Define
$$
P_1x = x -(x_1+x_2)(1,0,0)
$$
and
$$
P_2x = x -(x_1+x_2)(0,1,0).
$$
It is clear that
$$
\|P_1\|_o =\|P_1\|_w = \|P_2\|_o = \|P_2\|_w=1
$$
and $ P_1 \neq P_2.$ Hence there is no strongly unique minimal projection in this case.
\begin{rem}
\label{dim4}
Theorem (\ref{stronguni}) cannot be generalized for real spaces $X$ of dimension $ n \geq 4.$
\end{rem}
 Indeed let $ X= l_{\infty}^{(n)},$ and let $ V = ker(f),$
where $ f=(0,f_2,...,f_n) \in l_1^{(n)}$ satisfies $f_i>0$ for $i=2,...,n,$  $ \sum_{i=2}^n f_i=1$ and $ f_i <1/2 $ for $i=1,...,n.$ It is known
(see \cite{bl}, \cite{OL}) that in this case
$$
\lambda(X,V) = inf \{ \|P\|: P \in \mathcal{P}(X,V) \} = 1 + (\sum_{i=2}^n f_i/(1-2f_i))^{-1} >1.
$$
By \cite{aag-cbl}, $ \lambda(X,V) = \lambda_w^{id_V}(X,V).$
Define for $ i=2,...,n$ $ y_i= (\lambda(X,V)-1)(1-2f_i).$
Let $ y =(y_1,...,y_n)$ and $ z=(0,y_2,...,y_n).$ Consider mappings $P_1, P_2$ defined by
$$
P_1x = x-f(x)y
$$
and
$$
P_2x= x - f(x)z
$$
for $x \in l_{\infty}^{(n)}.$
It is easy to see that $ P_i \in \mathcal{P}(X,V),$ for $i=1,2,$ $ P_1 \neq P_2.$ By (\cite{OL} p. 104) $ \|P_i\|_o = \|P_i\|_w = \lambda(X,V)= \lambda_w^{id_V}$
for $i=1,2.$
\begin{rem}
\label{complex}
Theorem (\ref{stronguni}) is not valid for complex three dimensional spaces.
\end{rem}

Let $ X = l_{\infty}^{(3)}$ (in the complex case) and let
$$
V = \{ z \in X : z_1+z_2+z_3=0\}.
$$
Let $ y = (1,1,1).$ We show that
$$
Pz = z - (\frac{z_1+z_2+z_3}{3})y
$$
is a minimal projection in $ \mathcal{P}(X,V)$ with respect to the numerical radius and that
$$
\|P\|_w = 4/3.
$$
Let $f= (1/3,1/3,1/3).$ It is easy to see that (compare with \cite{OL}, p.103)
$$
\|P\| = max\{|(Pz)_j)|,j=1,2,3, \|z\|_{\infty}=1\}
$$
$$
= max\{|1-f_jy_j|+y_j(1-f_j):j=1,2,3\}= 4/3.
$$
Note that for $j=1,2,3$
$$
(e_j \otimes x^j)P = (Px_j)_j=4/3,
$$
where $ x^1=(1,-1,-1),$ $ x^2=(-1,1,-1)$ and $ x^3=(-1,-1,1).$ Since $e_j(x^j)=1$ for $j=1,2,3,$ we have  $\|P\|_w = 4/3.$
Also it is easy to see that
$$
W_P = \{ (e_j,x^j):j=1,2,3\}.
$$
Notice that
$$
\sum_{j=1}^3(e_j \otimes x^j)|_{B_V(X,V)}=0
$$
By Theorem (\ref{characterization}) and Remark (\ref{important2}), it follows that $0$ is a best approximation to $P$ in $B_V(X,V)$ with respect to the numerical radius,
which means that $P$ is a minimal projection with respect to the numerical radius.

Now define $ z = i(1,1,-2)$ and let $L= f(\cdot)z.$ It is clear that $ L \in B_V(X,V).$ Note that for $j=1,2,3$
$$
re((e_j \otimes x^j)L)= re(f(x^j)z_j) = f(x^j)re(z_j)=0.
$$
By Theorem (\ref{stronguni}), $0$ is not a SUBA to $P$ in $B_V(X,V),$ which proves our claim.

\noindent
\mbox{~~~~~~~}Asuman G\"{u}ven AKSOY\\
\mbox{~~~~~~~}Claremont McKenna College\\
\mbox{~~~~~~~}Department of Mathematics\\
\mbox{~~~~~~~}Claremont, CA  91711, USA \\
\mbox{~~~~~~~}E-mail aaksoy@cmc.edu \\ \\
\noindent
\mbox{~~~~~~~}Grzegorz LEWICKI\\
\mbox{~~~~~~~}Jagiellonian University\\
\mbox{~~~~~~~}Department of Mathematics\\
\mbox{~~~~~~~}\L ojasiewicza 6, 30-348, Krak\'ow, Poland\\
\mbox{~~~~~~~}E-mail: Grzegorz.Lewicki@im.uj.edu.pl\\\\

\end{document}